\documentclass{amsart}
\usepackage[utf8]{inputenc}
\usepackage{amsmath}
\usepackage{amsfonts} 
\usepackage{amsthm}
\usepackage{verbatim}
\usepackage{etoolbox}
\usepackage{dirtytalk}
\usepackage{tikz-cd}
\usepackage{derivative}
\usepackage{amssymb}
\usepackage{extarrows}

\theoremstyle{plain}
\newtheorem{thm}{Theorem}[section]
\newtheorem{lem}[thm]{Lemma}
\newtheorem{obs}[thm]{Observation}
\newtheorem{prop}[thm]{Proposition}
\newtheorem{cor}[thm]{Corollary}

\newtheorem*{claim*}{Claim}
\newtheorem{que}[thm]{Question}

\theoremstyle{definition}
\newtheorem{defi}[thm]{Definition}

\newtheorem{defiandque}[thm]{Definition and Question}

\theoremstyle{remark}

\numberwithin{equation}{section}

\AtBeginEnvironment{thm}{\setcounter{case}{0}}
\AtBeginEnvironment{lem}{\setcounter{case}{0}}
\AtBeginEnvironment{obs}{\setcounter{case}{0}}
\AtBeginEnvironment{defi}{\setcounter{case}{0}}
\AtBeginEnvironment{prop}{\setcounter{case}{0}}
\AtBeginEnvironment{cor}{\setcounter{case}{0}}

\AtBeginEnvironment{thm}{\setcounter{claim}{0}}
\AtBeginEnvironment{lem}{\setcounter{claim}{0}}
\AtBeginEnvironment{obs}{\setcounter{claim}{0}}
\AtBeginEnvironment{defi}{\setcounter{claim}{0}}
\AtBeginEnvironment{prop}{\setcounter{claim}{0}}
\AtBeginEnvironment{cor}{\setcounter{claim}{0}}

\newcommand{\C}{\mathbb{C}}
\newcommand{\K}{\mathbb{K}}

\newcommand{\PP}{\mathcal{P}}
\newcommand{\I}{\mathcal{I}}
\newcommand{\Z}{\mathcal{Z}}

\newcommand{\id}{\textnormal{id}}

\DeclareMathOperator{\codim}{codim}

\title{Retract rational varieties are uniformly retract rational}
\author{Juliusz Banecki}

\address{Faculty of Mathematics and Computer Science,
Jagiellonian University, ul. Lojasiewicza 6, 30-348 Krakow, Poland}
\email{juliusz.banecki@student.uj.edu.pl}
\date{}
\subjclass[2000]{14E08,	14M20}
\keywords{retract rational variety, Gromov elliptic variety, uniformly rational variety}

\begin{document}

\begin{abstract}
We prove that nonsingular retract rational algebraic varieties over any infinite field are uniformly retract rational. As a consequence, every rational, projective, nonsingular complex variety is algebraically elliptic.
\end{abstract}
\maketitle

\section{Introduction}
Throughout this paper, let $\K$ denote an infinite field. An affine variety $X$ over $\K$ is a Zariski closed subset of $\K^n$ for some $n$. By $\I(X)$ we denote the ideal of polynomials in $\K[x_1,\dots,x_n]$ vanishing on $X$, and by $\mathcal{P}(X)$ we denote the coordinate ring $\K[x_1,\dots,x_n]/\I(X)$ of $X$. Given a point $x\in X$, by $\mathcal{P}(X)_{\mathfrak{m}_{x}}$ or just by $\PP(X)_{x}$ we denote the localisation of $\mathcal{P}(X)$ at the maximal ideal at the point $x$, which can be interpreted as the ring of germs of regular functions at $x$. We say that the point $x$ is nonsingular if the local ring at $x$ is regular. If $X$ is irreducible, we denote the field of rational functions on $X$ by $\K(X)$, it is canonically isomorphic to the field of fractions of $\mathcal{P}(X)$. We say that a rational function is regular at a point $x\in X$ if it belongs to $\mathcal{P}(X)_{x}$ as an element of the field of fractions of the local ring.

Likewise, by an algebraic variety over $\K$ we mean a Zariski locally closed subset of the projective space $\K\mathbb{P}^n$ over $\K$ for some $n$. In the case that $\K$ is the field of complex (resp. real) numbers, we will say that the variety is complex (resp. real).

We can now begin with some historical background on the study of certain classes of algebraic varieties, started by Gromov. In his foundational paper \cite{Gromov}, he introduced the following notion, which is now called \emph{algebraic Gromov ellipticity}:
\begin{defi}
Let $X$ be a complex nonsingular algebraic variety. The variety $X$ is said to be \emph{algebraically elliptic} if there exists an algebraic vector bundle $E$ over $X$ and a regular mapping $s:E\rightarrow X$ such that
\begin{enumerate}
    \item $s(0_x)=x$ for $x\in X$, where $x\mapsto 0_x$ is the zero section,
    \item the derivative of $s\vert_{E_x}$ is surjective at $0_x$ as a mapping into $T_x X$, where $E_x$ is the fiber of the point $x$.
\end{enumerate}
\end{defi}

Gromov was able to provide some examples of complex algebraic varieties which are algebraically elliptic. In particular he noted that in the projective case, algebraic ellipticity follows from \emph{uniform rationality}, defined as follows:
\begin{defi}
An algebraic variety $X$ over an infinite field $\K$ of dimension $n$ is called \emph{uniformly rational} if for every point $x\in X$ there exists a Zariski open neighbourhood $V$ of $x$ which is biregularly isomorphic to a Zariski open subset of $\K^n$.
\end{defi}

In this context it is then natural to ask the following question:
\begin{que}[{\cite[p. 883 and 885]{Gromov}}]\label{que_Gromov}
Are all nonsingular rational algebraic varieties uniformly rational? In particular are all nonsingular rational complex projective varieties algebraically elliptic?
\end{que}
Some consequences of a positive answer to the first part of the question were analysed in \cite{Bogomolov}. The question has been restated in different versions in \cite{Zaidenberg3,Zaidenberg2,Zaidenberg_survey}.

Similar problems have also been considered in the real case. In the paper \cite{real_homogeneous}, Bochnak and Kucharz introduced the notion of \emph{malleable} varieties, which is the real counterpart to the class of algebraically elliptic varieties:
\begin{defi}
Let $X$ be a nonsingular real algebraic variety. The variety $X$ is said to be malleable if there exists an algebraic vector bundle $E$ over $X$ (see \cite[Definition 12.1.6]{RAG} for the definition of an algebraic vector bundle in the real case), and a regular mapping $s:E\rightarrow X$ such that
\begin{enumerate}
    \item $s(0_x)=x$ for $x\in X$, where $x\mapsto 0_x$ is the zero section,
    \item the derivative of $s\vert_{E_x}$ is surjective at $0_x$ as a mapping into $T_x X$, where $E_x$ is the fiber of the point $x$.
\end{enumerate}
\end{defi}

Building on the work of Bochnak and Kucharz, the author introduced the following definition in \cite{RelativeSW}:
\begin{defi}
An algebraic variety over an infinite field $\K$ is called \emph{uniformly retract rational} if for every point $x\in X$ there exists a Zariski open neighbourhood $V$ of $x$, a natural number $m$, a Zariski open set $U\subset \K^m$ and two regular mappings
\begin{equation*}
    V\rightarrow U\rightarrow V 
\end{equation*}
such that their composition is the identity on $V$.
\end{defi}
This definition begins to play an important role in the study of homotopy and approximation properties of regular mappings between real algebraic varieties as seen in the papers \cite{RelativeSW,bilski+kucharz,spaces_of_maps}.

With the existing machinery it is easy to prove the following
\begin{obs}\label{obs_urr=>Ge}
Complex projective (resp. real algebraic)  uniformly retract rational varieties are algebraically elliptic (resp. malleable).
\end{obs}

We postpone the proof of Observation \ref{obs_urr=>Ge} until the end of this section. 

Naturally, one can ask an analogue of Question \ref{que_Gromov} in the retract rational case:
\begin{defiandque}
An algebraic variety $X$ over an infinite field $\K$ is called retract rational if there exists a Zariski open and dense subset $V\subset X$, a natural number $m$, a Zariski open set $U\subset \K^m$ and two regular mappings
\begin{equation*}
    V\rightarrow U\rightarrow V 
\end{equation*}
such that their composition is the identity on $V$.

\emph{Are all nonsingular retract rational algebraic varieties uniformly retract rational?}
\end{defiandque}

The goal of this paper is to answer the question affirmatively:
\begin{thm}\label{main_thm}
Let $X$ be a nonsingular retract rational algebraic variety over an infinite field $\K$. Then it is uniformly retract rational.
\end{thm}

As a corollary we obtain a partial answer to Question \ref{que_Gromov}:
\begin{cor}
Nonsingular rational complex projective varieties are algebraically elliptic.
\begin{proof}
Clearly, rational varieties are retract rational. Then according to Theorem \ref{main_thm} they are uniformly retract rational, and hence thanks to Observation \ref{obs_urr=>Ge} they are algebraically elliptic.
\end{proof}
\end{cor}

In the end, the known implications between the different properties of an irreducible nonsingular complex (resp. real) projective variety $X$ now look as follows:
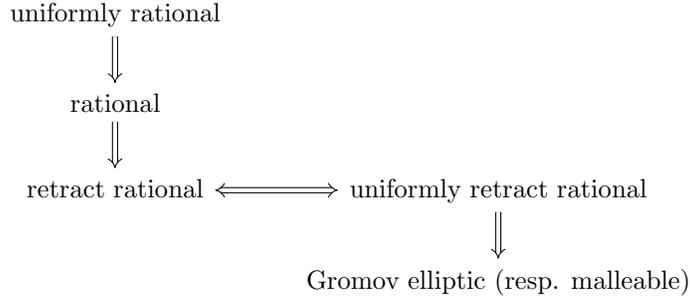
\begin{center}
\begin{tikzcd}
\text{uniformly rational} \arrow[d, Rightarrow]                                             &                                                         \\
\text{rational} \arrow[d, Rightarrow]                                                       &                                                         \\
\text{retract rational} \arrow[r, Leftrightarrow] & \text{uniformly retract rational} \arrow[d, Rightarrow] \\
                                                                                            & \text{algebraically elliptic (resp. malleable)}                     
\end{tikzcd}
\end{center}
See also the diagram in \cite[p. 5]{RelativeSW} to get a more complete picture in the real case.

It is worth noting that there exist complex projective algebraically elliptic varieties which are not rational. Indeed, by \cite[Theorem 1.1]{cubic}, all nonsingular cubic hypersurfaces in $\C\mathbb{P}^4$ are algebraically elliptic, while their irrationality follows from the classical result of Clemens and Griffiths \cite{clemens-griffiths}.

There also exists an example of a complex affine variety that is algebraically elliptic but not retract rational. In \cite[Theorem 3.6]{Saltman1984}, Saltman constructed a homogeneous space of the group $\mathrm{SL}(n,\C)$ for some $n$, which is not retract rational. As noted in \cite[p. 2]{cubic}, this variety is nonetheless algebraically elliptic.

\begin{proof}[Proof of Observation \ref{obs_urr=>Ge}]
In the real case this has already been established in \cite[Theorem 2.11]{RelativeSW}.

Let then $X$ be a complex projective uniformly retract rational variety. Fix a point $x_0\in X$ and its Zariski open neighbourhood $V\subset X$ such that there exists a Zariski open set $U\subset \C^m$ and two regular mappings $i:V\rightarrow U,r:U\rightarrow V$ such that $r\circ i=\id_V$. Since $X$ is projective, the mapping $r$ admits an extension to a regular mapping $r':U'\rightarrow X$ defined on a Zariski open set $U'\subset\C^m$, with $\codim \C^m\backslash U'\geq 2$. Then according to \cite[Proposition 6.4.1]{forstneric}, $U'$ is an algebraically elliptic variety. Let $E$ be an algebraic vector bundle over $U'$ with a dominating spray $s:E\rightarrow U'$. We define a regular mapping $s'$ from the pulled back bundle $E':=i^\ast E$ into $X$ by
\begin{equation*}
    s'(v,x):=r'\circ s(v).
\end{equation*}
It is clear from the definition that 
\begin{enumerate}
    \item $s'(0_x)=x$ for $x\in V$,
    \item the derivative of $s'\vert_{E'_{x_0}}$ at the point $0_{x_0}$ is surjective.
\end{enumerate}
 
A triple $(E',s',V)$ satisfying these properties is called a local dominating spray at $x_0$. According to \cite[Theorem 1.1]{Zaidenberg2}, a variety admitting a local dominating spray at each of its points is algebraically elliptic. As the point $x_0$ was arbitrary the conclusion follows.
\end{proof}

\section{Preliminaries}
From now on let $\K$ be a fixed infinite field. We denote by $\overline{\K}$ its algebraic closure. Given a closed subvariety $X$ of $\K^n$, by $X^{\overline{\K}}$ we denote the Zariski closure of $X$ in $\overline{\K}^n$. Given an ideal $I\subset \PP(X)$, by $\Z(I)$ we denote its zero set in $X$, while by $\Z^{\overline{\K}}(I)$ we denote the zero set in $X^{\overline{\K}}$ of its extension in the ring $\PP(X^{\overline{\K}})$. Given two closed subvarieties $X\subset \K^n,Y\subset \K^m$ of $\K^n$ and $\K^m$ respectively and a polynomial mapping $f:X\rightarrow Y$, we denote by $f^{\overline{\K}}:X^{\overline{\K}}\rightarrow Y^{\overline{\K}}$ its natural polynomial extension. Given a variety $X$, we say that a property is true for a generic point $x\in X$ if it is satisfied on a Zariski open and dense subset of $X$.

We begin with the following observation:
\begin{obs}\label{obs_flat}
Let $X\subset \K^n, Y\subset \K^m$ be two nonsingular affine varieties over $\K$. Let $\sigma:X\rightarrow Y$ be a regular mapping such that its derivative at a point $x_0\in X$ is an isomorphism (considered as a mapping between the Zariski tangent spaces of $X$ at $x_0$ and $Y$ at $\sigma(x_0)$). Then, the induced mapping 
\begin{equation*}
    \PP(Y)_{\sigma(x_0)}\rightarrow\PP(X)_{x_0}
\end{equation*}
is injective. Moreover, if $P,Q\in\PP(Y)_{\sigma(x_0)}$ are such that $Q\circ \sigma$ divides $P\circ \sigma$ in the ring $\PP(X)_{x_0}$, then $Q$ divides $P$ in $\PP(Y)_{\sigma(x_0)}$.
\begin{proof}
It follows from the assumption  that $\sigma$ induces an isomorphism between the formal completions of $\PP(Y)_{\sigma(x_0)}$ and $\PP(X)_{x_0}$. Injectivity follows from the fact that the natural map from $\PP(Y)_{\sigma(x_0)}$ to its completion is injective. The second part follows by a similar argument; if $Q\circ \sigma$ divides $P\circ \sigma$ in $\PP(X)_{x_0}$, then $Q$ divides $P$ in the completion of $\PP(Y)_{\sigma(y_0)}$, and as the completion is flat over $\PP(Y)_{\sigma(x_0)}$ the quotient must belong to $\PP(Y)_{\sigma(x_0)}$.
\end{proof}
\end{obs}

The following proposition will play an important role in our proof of Theorem \ref{main_thm}:
\begin{prop}\label{alg_prop}
Let $X\subset \K^n$ be a Zariski closed irreducible nonsingular subvariety of $\K^n$ of dimension $m$, and let $x_0\in X$. Let $I\subset \PP(X)$ be a nonzero ideal such that $I\subset \mathfrak{m}_{x_0}$, by the same symbol we denote the ideal generated by $I$ in the ring $\PP(X)[t_1,\dots,t_{n-m}]$, where $t_1,\dots,t_{n-m}$ are new variables. Then, there exists a polynomial mapping $\sigma:X\times \K^{n-m}\rightarrow \K^n$, such that
\begin{enumerate}
    \item $\sigma(x,0)=x$ for $x\in X$,
    \item the derivative of $\sigma$ at $(x_0,0)$ is an isomorphism,
    \item the homomorphism induced by pre-composing with $\sigma$ and then taking the quotient modulo the product ideal $I(t_1,\dots,t_{n-m})$
    \begin{equation*}
        \PP(\K^n)_{x_0}\rightarrow (\PP(X)[t_1,\dots,t_{n-m}]/I(t_1,\dots,t_{n-m}))_{\mathfrak{m}_{(x_0,0)}}
    \end{equation*}
    is surjective.
\end{enumerate}
\begin{proof}
Without loss of generality we can assume $x_0=0\in\K^n$. Let $W$ be a generic vector subspace of $\K^n$ of codimension $m$.

We can assume that it intersects $X$ transversally at zero. Moreover, by a simple dimension counting argument (see Lemma \ref{generic} below), we can assume that $\Z^{\overline{\K}}(I)\cap W^{\overline{\K}}=\{0\}$. 

Let $x_1,\dots,x_n$ be a basis of the dual space of $\K^n$ such that $W$ is given by the equation $x_{n-m+1}=\dots=x_n=0$. From Noether normalisation lemma we can assume that
\begin{equation}\label{finite_ext_eq}
	\PP(X) \text{ is integral over }\K[\bar{x}_1,\dots,\bar{x}_m],
\end{equation} 
where $\bar{x}_i$ denotes the restriction $x_i+\I(X)\in\PP(X)$ of $x_i$ to $X$.

For $1\leq i\leq k$ let $t_i:=x_i+\I(W)\in\PP(W)$. The linear system of parameters $t_1,\dots,t_{n-m}$ allows for a natural identification of $W$ with $\K^{m-n}$. Under this identification define
\begin{align*}
	&\sigma:X\times W\rightarrow \K^n,\\
	&\sigma(x,v):=x+v.
\end{align*}  
It is clear that under such a definition of $\sigma$, the properties $(1)$ and $(2)$ are satisfied.

Consider the ring extension $\PP(\K^n)\subset \PP(X\times W)$ induced by $\sigma$. Identifying $\PP(X)$ with its image under the natural inclusion $\PP(X)\subset \PP(X\times W)$ for $n-m< i \leq n$ we have that $\bar{x}_i= x_i\in \PP(\K^n)$, so from \eqref{finite_ext_eq} we deduce that all the elements of $\PP(X)$ are integral over $\PP(\K^n)$. Since for $1\leq i \leq m$ we have $t_i+\bar{x}_i= x_i\in \PP(\K^n)$, it follows that the entire ring $\PP(X\times W)$ is integral over $\PP(\K^n)$. As it is finitely generated as an algebra, it is finitely generated as a module over $\PP(\K^n)$. 

Denote by $\mathfrak{n}$ the maximal ideal of functions vanishing at the point $(0,0)$ in $\PP(X\times W)$ and by $\mathfrak{m}=\mathfrak{n}\cap \PP(\K^n)$ the maximal ideal of functions vanishing at the point $0$ in $\PP(\K^n)$. 

Note that if $(x,v)\in \Z^{\overline{\K}}(I(t_1,\dots,t_{n-m}))\subset X^{\overline{\K}}\times W^{\overline{\K}}$ is a point with $\sigma^{\overline{\K}}(x,v)=0$, then either $x\in \Z^{\overline{\K}}(I)\cap W^{\overline{\K}}=\{0\}$, so $x=0$ and $v=0$, or $v\in \Z^{\overline{\K}}(t_1,\dots,t_m)=\{0\}$, so again $v=0$ and $x=0$. This shows that 
\begin{equation*}
    \Z^{\overline{\K}}(\mathfrak{m}\PP(X\times W) +I(t_1,\dots,t_{n-m}))=\{(0,0)\}.
\end{equation*}
In algebraic terms this means that $\mathfrak{n}$ is the only associated prime of $\mathfrak{m}\PP(X\times W) +I(t_1,\dots,t_{n-m})$.

Now, from point $(2)$ we deduce that elements of $\mathfrak{m}$ generate the maximal ideal of the local ring $\PP(X\times W)_{\mathfrak{n}}$. This means that for every $f\in\mathfrak{n}$ there is $q\in\PP(X\times W)\backslash\mathfrak{n}$ such that
\begin{equation*}
	qf\in \mathfrak{m}\PP(X\times W)\subset \mathfrak{m}\PP(X\times W) +I(t_1,\dots,t_{n-m}).
\end{equation*}
As the latter ideal is $\mathfrak{n}$-primary we get that 
\begin{equation}\label{ideal_equal_eq}
	\mathfrak{m}\PP(X\times W) +I(t_1,\dots,t_{n-m})=\mathfrak{n}.
\end{equation}

The ideal $I(t_1,\dots,t_{n-m})$ is a submodule of $\PP(X\times W)$. Consider the quotient module $M:=\PP(X\times W)/(I(t_1,\dots,t_{n-m})+\PP(\K^n))$. From \eqref{ideal_equal_eq} we deduce that 
\begin{equation*}
	M\subset M\mathfrak{m}.
\end{equation*}
Hence, it follows from Nakayama's lemma that $M_{\mathfrak{m}}=0$. This is equivalent to $(3)$.
\end{proof}
\end{prop}

\begin{lem}\label{generic}
Let $X\subset \K^n$ be a Zariski closed irreducible subvariety of $\K^n$ of dimension $m$ containing the origin and let $I\subset \PP(X)$ be a nonzero ideal with $0\in\Z(I)$. Then, a generic element $W$ of the Grassmannian $\mathrm{Gr}_{n-m}(\K^n)$ satisfies
\begin{equation*}
   W^{\overline{\K}}\cap \Z^{\overline{\K}}(I)=\{0\}.  
\end{equation*}
\begin{proof}
It suffices to consider the case when $\K$ itself is algebraically closed, since $\mathrm{Gr}_{n-m}(\K^n)$ is irreducible and it is Zariski dense in $\mathrm{Gr}_{n-m}(\overline{\K}^n)$.

Let $Z:=\Z(I)$ and consider the variety
\begin{equation*}
    Y:=\{(V,x)\in\mathrm{Gr}_{n-m}(\K^n)\times (Z\backslash \{0\}):x\in V\}. 
\end{equation*}
Each fiber under the natural projection $\pi_2:Y\rightarrow Z\backslash \{0\}$ is isomorphic to the lower dimensional Grassmannian $\mathrm{Gr}_{n-m-1}(\K^{n-1})$. Since $I$ is nonzero it folows that $\dim Z\leq m-1$, so 
\begin{equation*}
    \dim Y\leq \dim \mathrm{Gr}_{n-m-1}(\K^{n-1})+\dim Z\leq(n-m)m-1<\dim \mathrm{Gr}_{n-m}(\K^n).
\end{equation*}
This shows that a generic element of $\mathrm{Gr}_{n-m}(\K^n)$ does not lie in the image of the projection $\pi_1:Y\rightarrow \mathrm{Gr}_{n-m}(\K^n)$, which finishes the proof.
\end{proof}
\end{lem}

Lastly, before proving Theorem \ref{main_thm}, we need the following crucial lemma:

\begin{lem}\label{denom_reduction}
Let $P,Q\in \PP(\K^n)$ two polynomials. Let $X$ be an affine nonsingular irreducible affine variety over $\K$, and let $\pi:X\rightarrow \K^n$ be a regular mapping, such that $Q\circ \pi$ is not identically equal to zero and $\frac{P\circ \pi}{Q\circ \pi}\in \PP(X)_{x_0}$ for some point $x_0\in X$. Let $\varphi:(X,x_0)\rightarrow \K^n$ be a germ of a regular mapping, such that for $1\leq i \leq n$ its $i$-th coordinate $\varphi_i$ satisfies
\begin{equation*}
    \varphi_i-\pi_i\in (Q\circ\pi)\mathfrak{m}_{x_0}\text{ in the ring }\PP(X)_{x_0}.
\end{equation*}
Then, $Q\circ\varphi$ is not identically equal to zero and the rational mapping $\frac{P\circ \varphi}{Q\circ \varphi}$ is regular at $x_0$.
\begin{proof}
Note that for $v,w\in\K^n$, the polynomials $P$ and $Q$ can be written as
\begin{align*}
    P(v+w)=P(v)+\sum_i w_i P_i(v,w),\\
    Q(v+w)=Q(v)+\sum_i w_i Q_i(v,w),
\end{align*}
for some $P_i,Q_i\in \PP(\K^n)$. Then
\begin{equation*}
    Q\circ\varphi(x)=Q\circ\pi(x)\left(1+\sum_i \frac{\varphi_i(x)-\pi_i(x)}{Q\circ\pi(x)} Q_i\big(\pi(x),\varphi(x)-\pi(x)\big)\right)
\end{equation*}
is a product of two nonzero elements of $\PP(X)_{x_0}$ and
\begin{equation*}
    \frac{P\circ\varphi(x)}{Q\circ\varphi(x)}=\frac{\frac{P\circ\pi(x)}{Q\circ\pi(x)}+\sum_i \frac{\varphi_i(x)-\pi_i(x)}{Q\circ \pi(x)} P_i\big(\pi(x),\varphi(x)-\pi(x)\big)}{1+\sum_i \frac{\varphi_i(x)-\pi_i(x)}{Q\circ\pi(x)} Q_i\big(\pi(x),\varphi(x)-\pi(x)\big)}\in\PP(X)_{x_0}.
\end{equation*}
\end{proof}
\end{lem}

\section{Proof of the main theorem}

Theorem \ref{main_thm} follows easily from the next proposition. The proposition may be of some interest on its own.
\begin{prop}\label{main_prop}
Let $X\subset \K^n,Y\subset \K^k$ be two Zariski closed nonsingular irreducible subvarieties of $\K^n$ and $\K^k$ respectviely. Let $x_0$ be a point of $X$ and $F:\K^n\dashrightarrow Y$ be a rational mapping, such that all its coordinates belong to $\PP(\K^n)_{\I(X)}$, so that the restriction $F\vert_X:X\dashrightarrow Y$ makes sense as a rational mapping. Assume further that the restricted morphism is regular at $x_0$. Then, there exists a regular germ $G:(\K^n,x_0)\rightarrow Y$ such that $G\vert_X=F\vert_X$ as rational mappings.  

\begin{proof}
Set $m:=\dim X$. Write each coordinate $F_i$ of $F$ as $\frac{P_i}{Q_i}$ with $P_i,Q_i\in\PP(\K^n)$ and $Q_i\not\in\I(X)$. Set $q:=Q_1\dots Q_k\vert_X\in\PP(X)$. Without loss of generality $q(x_0)=0$, for otherwise we can take $G:=F$. Applying Proposition \ref{alg_prop} to the ideal $I=(q)$, we find a polynomial mapping $\sigma:X\times\K^{n-m}\rightarrow \K^n$ such that
\begin{enumerate}
    \item $\sigma(x,0)=x$ for $x\in X$,
    \item the derivative of $\sigma$ at $(x_0,0)$ is an isomorphism,
    \item the induced homomorphism
    \begin{equation*}
        \PP(\K^n)_{x_0}\rightarrow (\PP(X\times \K^{n-m})/(q)(t_1,\dots,t_{n-m}))_{\mathfrak{m}_{(x_0,0)}}
    \end{equation*}
    is surjective.
\end{enumerate}

Let $\pi:X\times\K^{n-m}\rightarrow \K^n$ be the projection $\pi(x,t)=x\in X\subset \K^n$. Since the image of $\pi$ is equal to $X$, from the assumptions on $F$ we have that for $1\leq i \leq k$ the polynomial $Q_i\circ \pi$ is not identically equal to zero and $\frac{P_i\circ \pi}{Q_i\circ \pi}\in\PP(X)_{x_0}$.

Using the property $(3)$ of $\sigma$, for $1\leq i\leq n$ we can find a germ of a regular function $\psi_i\in\PP(\K^n)_{x_0}$ such that
\begin{equation*}
    \psi_i\circ\sigma-\pi_i\in (q)(t_1,\dots,t_{n-m})\text{ in the ring }\PP(X\times\K^{n-m})_{(x_0,0)}.
\end{equation*}
These together give a regular germ $\psi:(\K^n,x_0)\rightarrow \K^n$.

If we now define $\varphi:(X\times\K^{n-m},(x_0,0))\rightarrow \K^n$ as $\varphi:=\psi\circ\sigma$, then from Lemma \ref{denom_reduction} we deduce that for $1\leq i \leq k$ we have that $Q_i\circ\varphi$ is not identically equal to zero and $\frac{P_i\circ\varphi}{Q_i\circ\varphi}\in\PP(X\times \K^{n-m})_{(x_0,0)}$. As the derivative of $\sigma$ is an isomorphism at $(x_0,0)$, from Observation \ref{obs_flat} we get that $Q_i\circ \psi$ is not identically equal to zero and $\frac{P_i\circ \psi}{Q_i\circ \psi} \in \PP(\K^n)_{x_0}$.

We can now define $G:\K^n\dashrightarrow Y$ as $G:=F\circ \psi$. We have just verified that $G$ is well defined and regular on a neighbourhood of $x_0$. 

Since all coordinates of $\psi\circ\sigma$ are congruent to the respective coordinates of $\pi$ modulo $(t_1,\dots,t_{n-m})$, it follows that for points $x\in X$ we have $\psi\circ \sigma(x,0)=\pi(x,0)$ whenever the left hand side is defined. Hence, if moreover $x$ is not contained in the zero set of $q$ then
\begin{equation*}
	G(x)=F\circ\psi(x)=F\circ\psi\circ\sigma(x,0)=F\circ\pi(x,0)=F(x).
\end{equation*}
This shows that $G\vert_X=F\vert_X$.
\end{proof}
\end{prop}

\begin{proof}[Proof of Theorem \ref{main_thm}]
Let $X$ be a nonsingular retract rational variety, and let $x_0$ be a point of $X$. Retract rational varieties are irreducible, and the property of being uniformly retract rational is local, so passing to a neighbourhood we can assume that $X$ is embedded in $\K^n$ as a Zariski closed irreducible subvariety. Let $V\subset X$ and $U\subset \K^m$ be two nonempty Zariski open sets with regular mappings $i:V\rightarrow U$ and $r:U\rightarrow V$, such that their composition $r\circ i$ is the identity on $V$. Treating each of the coordinates separately, $i$ can be extended to a mapping $i_0:W\rightarrow U$, where $W$ is a Zariski open subset of $\K^n$ containing $V$ as a subset. Applying Proposition \ref{main_prop} to the rational mapping $F:=r\circ i_0:\K^n\dashrightarrow X$, we find a regular germ 
\begin{equation*}
    G:(\K^n,x_0)\rightarrow X
\end{equation*}
such that $G\vert_X$ is equal to the identity as a germ at $x_0$. Finally, taking $W'$ to be some Zariski open neighbourhood of $x_0$ on which $G$ is defined and setting $V':=X\cap W'$ and $U'=G^{-1}(V')$, we get that $G\vert_{U'}$ is a regular retraction from $U'$ to $V'$.
\end{proof}

\bibliographystyle{plain}
\bibliography{references}
\end{document}